\documentclass[12pt,reqno]{amsart}
\usepackage{amssymb}

\newtheorem{theorem}{Theorem}[section]
\newtheorem{lemma}[theorem]{Lemma}
\newtheorem{proposition}[theorem]{Proposition}
\newtheorem{corollary}[theorem]{Corollary}

\theoremstyle{definition}

\newtheorem{remark}[theorem]{Remark}

\numberwithin{equation}{section}

\begin{document}

\title[Subalgebras of vector fields on the plane]{On Lie's
classification of subalgebras of vector fields on the plane}

\author[H. Azad]{Hassan Azad}

\address{Abdus Salam School of Mathematical Sciences, GC University Lahore, 68-B, New Muslim
Town, Lahore 54600, Pakistan}

\email{hassan.azad@sms.edu.pk, hssnazad@gmail.com}

\author[I. Biswas]{Indranil Biswas}

\address{School of Mathematics, Tata Institute of Fundamental
Research, Homi Bhabha Road, Mumbai 400005, India}

\email{indranil@math.tifr.res.in}

\author[F. M. Mahomed]{Fazal M. Mahomed}

\address{DST-NRF Centre of Excellence in Mathematical and Statistical Sciences,
School of Computer Science and Applied Mathematics,
University of the Witwatersrand, Johannesburg Wits 2050, South Africa and\\
School of Mathematics and Statistics, The University of New South Wales, Sydney
NSW 2052 Australia}

\email{Fazal.Mahomed@wits.ac.za}

\author[S. W. Shah]{Said Waqas Shah}

\address{Abdus Salam School of Mathematical Sciences, GC University Lahore, 68-B, New Muslim
Town, Lahore 54600, Pakistan}

\email{waqas.shah@sms.edu.pk}

\thanks{FM is Visiting Professor at UNSW for 2020}



\date{}

\begin{abstract}
A brief proof of Lie's classification of solvable algebras of vector fields on the plane is given.
The proof uses basic representation theory and PDEs. 
\end{abstract}

\maketitle

\section{Introduction}\label{se1}

The aim of this note is to give a brief proof of Lie's classification of all solvable Lie algebras of vector fields on
the plane. The main results are given as Theorem \ref{prop2}, Corollary \ref{cor6.4},
Proposition \ref{pro3} and Proposition \ref{propa}. A sketch of the arguments is given in
Section \ref{se2}. The original proofs in \cite{Li1}, \cite{Li2} classify all algebras of vector fields in
$\mathbb{C}^{2}$ in terms of transitive, primitive and imprimitive algebras.
Later this was extended by the authors of \cite{GKO} to vector fields in ${\mathbb R}^2$.
We will first classify nilpotent
nonabelian algebras and then solvable algebras that are not nilpotent, of vector fields defined on open subsets of
${\mathbb C}^2$, up to locally defined $C^{\infty}$ change of variables; see \cite[p.~3--4]{Li2} for further
explanation on this.

For solvable algebras which are not nilpotent we first work over
${\mathbb C}$ and derive the results over ${\mathbb R}$ as
consequences of the classification over ${\mathbb C}$.

In \cite{SMAM} a complete
classification of all 4th-order equations, according to the structure of their symmetry algebras, is given.
The main tool used there is Lie's classification. Moreover, as the Borel algebra of a semisimple Lie
algebra essentially determines the full Lie algebra, a knowledge of a detailed list of solvable algebras is useful in
determining the corresponding semisimple algebras. For this reason, it is desirable to look at Lie's classification
from different points of view.

\section{Sketch of the arguments}\label{se2}

For the sake of convenience, we call a Lie algebra of vector fields defined on some open subset of
${\mathbb C}^2$ as a Lie algebra of vector fields on ${\mathbb C}^2$. The vector fields
$\frac{\partial}{\partial x}$ and $\frac{\partial}{\partial y}$ will be denoted by ${\partial}_x$
and ${\partial}_y$, respectively.

In this paper, by the {\it rank} of a Lie algebra of vector fields we mean the dimension of its generic orbit (this
should not be confused with the
dimension of its Cartan subalgebra). It should be mentioned that
for semisimple Lie algebras of vector fields, the rank and dimension of their
Cartan subalgebras coincide (cf. \cite{ABM}).

We first classify nilpotent, nonabelian algebras of $C^{\infty}$ vector fields in the (real or complex) plane up to 
locally defined change of variables --- the abelian case being easy. Up to point transformations, there is only one and 
its center is of dimension one while the algebra is of rank two.

Any solvable Lie algebra $\mathfrak{g}$ of $C^\infty$ vector fields in the plane,
whose commutator is non-abelian, operates on the one dimensional center of
$\mathfrak{g}'\,:=\, [\mathfrak{g},\,\mathfrak{g}]$. This determines the possible forms of representatives of
$\mathfrak{g}/\mathfrak{g}'$. Their
representatives are brought to their simplest form by solving some elementary partial differential equations.

When $\mathfrak{g}'$ is of rank two, $\mathfrak{g}$ is contained
in its normalizer, which consists of linear fields and the classification in this case is straightforward.

When the commutator of $\mathfrak{g}$ is abelian and it is of rank one, we need to first assume that $\mathfrak{g}$ is
defined over ${\mathbb C}$ so that Lie's theorem --- on solvable algebras fixing a line in any representation --- can
be used. This determines the form of representatives of $\mathfrak{g}/\mathfrak{g}'$ and the representatives are
brought to their simplest form by solving elementary partial differential equations. It turns out that
$\mathfrak{g}/\mathfrak{g}'$ is of dimension at most two. There is only one Jordan block for each eigenvalue of
$\mathfrak{g}/\mathfrak{g}'$
on $\mathfrak{g}'$. If now $\mathfrak{g}$ is a real algebra and $\mathfrak{g}^{\mathbb C}$ its complexification, then as
$(\mathfrak{g}/\mathfrak{g}')^{\mathbb C} \,=\, \mathfrak{g}^{\mathbb C}/(\mathfrak{g}')^{\mathbb C}$,
it follows that $\mathfrak{g}/ \mathfrak{g}'$ is also of dimension
at most two.

The real case of Theorem \ref{prop6.5} follows by applying Lie's Theorem
on real solvable algebras; this is explained in Remark \ref{re6.4}.

\section{Preliminary lemmas}

Although the field in this section is $\mathbb C$, it can be replaced by $\mathbb R$.

We will consider $C^\infty$ vector fields defined on an open subset of $\mathbb C$. Let $V({\mathbb C})$ denote the
space of all such vector fields.

All vector fields considered here are $C^\infty$.

\begin{lemma}\label{lem1}
Let $\mathfrak g$ be a Lie algebra of vector fields on
an open subset of ${\mathbb C}^2$ such that every
element of $\mathfrak g$ is of the form $\xi(x,y)\partial_x+\eta(y)\partial_y$ and
$$
\Pi\, :\, {\mathfrak g}\, \longrightarrow\, V({\mathbb C})
$$
be the map defined by $\xi(x,y)\partial_x+\eta(y)\partial_y\, \longrightarrow\, \eta(y)\frac{d}{dy}$.
Then $\Pi$ is a homomorphism of Lie algebras.

Moreover, if $\mathfrak g$ is finite dimensional and solvable, then
$\Pi(\mathfrak g)$, after a change of variables involving only $y$, is
a subalgebra of the algebra
\begin{equation}\label{h1}
\{(a+by)\frac{d}{dy}\, \in\, V({\mathbb C})\, \mid\, a,\, b\, \in\,\mathbb C\}\, .
\end{equation}
\end{lemma}

\begin{proof}
It is evident that $\Pi$ is a homomorphism of Lie algebras.

The commutator of a finite dimensional nonabelian solvable Lie algebra is nilpotent and has a
nonzero center. The second part follows from this and the fact that the normalizer of
$\langle \frac{d}{dy}\rangle \, \subset\, V({\mathbb C})$ is the two dimensional
subalgebra in \eqref{h1}.
\end{proof}

\begin{lemma}\label{lem2}
Let $\mathfrak g$ be a finite dimensional abelian Lie algebra of vector fields on
an open subset of ${\mathbb C}^2$. Then one of the following two holds:
\begin{enumerate}
\item ${\rm rank}({\mathfrak g})\,=\, 1$.

\item ${\rm rank}({\mathfrak g})\,=\, 2$ and ${\mathfrak g}\,=\, \langle \partial_x,\, \partial_y\rangle$
up to point transformations.
\end{enumerate}
\end{lemma}

\begin{proof}
If ${\rm rank}({\mathfrak g})\,=\, 2$, then $\mathfrak g$ has a subalgebra equivalent to
$\langle \partial_x,\, \partial_y\rangle$ and thus
${\mathfrak g}\,=\, \langle \partial_x,\, \partial_y\rangle$. From this it follows that
${\rm rank}({\mathfrak g})\,\leq\, 2$.
\end{proof}

\section{Nilpotent non-abelian subalgebras}

In this section $\mathbb C$ can be replaced by $\mathbb R$.

\begin{proposition}\label{prop1}
Let $\mathfrak g$ be a nilpotent non-abelian Lie algebra of vector fields on ${\mathbb C}^2$.
Then the dimension of the center ${\mathfrak z}({\mathfrak g})$ of $\mathfrak g$ is one.
\end{proposition}

\begin{proof}
The Lie algebra $\mathfrak g$ has a nontrivial center. If the rank of the center is two, then
necessarily ${\mathfrak g}\,=\, \langle \partial_x,\, \partial_y\rangle$ and $\mathfrak g$ will
be abelian. But $\mathfrak g$ is non-abelian.

If the dimension of the center ${\mathfrak z}({\mathfrak g})$ of ${\mathfrak g}$
is two or more, then in suitable coordinates
$$
{\mathfrak z}({\mathfrak g})\, \supset\, \langle \partial_x,\, y\partial_x\rangle\, .
$$
Now every element of ${\mathfrak g}$ is of the form
$P(y)\partial_x+ Q(y)\partial_y$. Moreover,
$$
[y\partial_x,\, P(y)\partial_x+ Q(y)\partial_y]\,=\, -Q(y)\partial_y\,=\, 0\, .
$$
Thus, we have $Q(y)\,=\ 0$. Hence, all elements of $\mathfrak g$ are of the form
$P(y)\partial_x$. But this again implies that $\mathfrak g$ is abelian, contradicting
that $\mathfrak g$ is non-abelian.
\end{proof}

\begin{theorem}\label{prop2}
Let $\mathfrak g$ be a nilpotent non-abelian Lie algebra of vector fields on ${\mathbb C}^2$.
Then the canonical form of $\mathfrak g$ --- in suitable coordinates --- is
$$
\langle \partial_y\rangle + \langle \partial_x,\, y\partial_x ,\,\cdots,\, y^{N}\partial_x\rangle\,
$$ for $N \geq 1$.
\end{theorem}

\begin{proof}
From Proposition \ref{prop1} we know that $\dim {\mathfrak z}({\mathfrak g})\,=\, 1$. Let
${\mathfrak z}({\mathfrak g})\,=\langle \partial_x\rangle$ in suitable coordinates $x,\, y$.
There must be an element $X\, \in\, \mathfrak g$ of the form
$P(y)\partial_x+ Q(y)\partial_y$ with $Q\, \not=\, 0$, since otherwise $\mathfrak g$ will be
abelian. It is straightforward to check that
$$
\langle \partial_x,\, P(y)\partial_x+ Q(y)\partial_y\rangle\, \subset\, \mathfrak g
$$
is an abelian subalgebra of rank two. Therefore, there is a system of coordinates
$\widetilde{x},\, \widetilde{y}$ such that $\partial_x\,=\, \partial_{\widetilde{x}}$
and $X\,=\, \partial_{\widetilde{y}}$.

Writing $\widetilde{x},\, \widetilde{y}$ as $x,\, y$ for convenience we see that every element of $\mathfrak g$
is of the form $\xi(y)\partial_x+ \eta(y)\partial_y$ and ${\mathfrak g}\, \supset\,\langle \partial_x,\,
\partial_y\rangle$. By Lemma \ref{lem1}, the projection of $\mathfrak g$ on the $y$--axis is a nilpotent algebra
of vector fields on $\mathbb C$ and it is a non-zero algebra. Hence this projection is $\langle\partial_y\rangle$
in a change of coordinates involving only $y$. Thus the elements of $\mathfrak g$ are of the form
$\xi(y)\partial_x+ \lambda\partial_y$ and $\mathfrak g$ contains $\langle \partial_x,\, \partial_y\rangle$.

As ${\rm ad}\, \partial_y$ on $\mathfrak g$ is nilpotent, we conclude that for all elements in $\mathfrak g$ of the form
$\xi(y)\partial_x$, the function $\xi$ must be a polynomial. Take such an element $P(y)\partial_x$
of largest degree $N$, say
$$
P(y)\partial_x\,=\, (\sum_{i=0}^N a_i y^i)\partial_x\, ,
$$
with $a_N\, \not=\, 0$. Applying ${\rm ad}\, \partial_y$ repeatedly, we see that all the
$N$ vector fields $y^{N}\partial_x,\,\cdots,\, y\partial_x$ lie in $\mathfrak g$. Thus,
${\mathfrak g}\,=\, \langle \partial_y\rangle + \langle \partial_x,\, y\partial_x ,\,\cdots,\,
y^{N}\partial_x\rangle$
in suitable coordinates. It is straightforward to check that the center of it is
$\langle \partial_x\rangle$.
\end{proof}

\section{Solvable non-nilpotent subalgebras of vector fields}

Assume that $\mathfrak{g}$ is solvable but not nilpotent. Using Proposition
\ref{prop1} it is deduced that one of the following two holds:
\begin{enumerate}
\item ${\mathfrak g}'$ is non-abelian and the dimension of its center is one.

\item ${\mathfrak g}'$ is abelian.
\end{enumerate}

\begin{proposition}\label{pro3}
If ${\mathfrak g}$ is solvable and $\mathfrak{g}'$ is not abelian, then
$$
{\mathfrak g}'\,=\, \langle \partial_x,\, \partial_y,\, y\partial_x,\, \cdots,\,
y^k\partial_x\rangle, \ \ k \,\geq\, 1,
$$
and either
$$
\mathfrak{g}\, =\, \langle x \partial_{x},\, y \partial_{y} \rangle + \langle\partial_x,\,\partial_y,\,
y\partial_x,\, \cdots,\, y^{k} \partial_x \rangle, \ \ k \,\geq\, 1,
$$
or
$$
\mathfrak{g} \,=\, \langle a x\partial_x + y\partial_y \rangle + \langle\partial_x,\, \partial_y,\,
y \partial_x, \,\cdots, \, y^{k}\partial_x \rangle,\ \ a \,\neq\, 0,\, k\, .
$$
\end{proposition}

\begin{proof}
The first statement is Theorem \ref{prop2}. The center of $\mathfrak{g}'$ is $\langle \partial_x \rangle$. Thus every element $X$ of $\mathfrak{g}$ normalizes $\langle \partial_x \rangle $
and therefore
$$
X = (\lambda x + H(y)) \partial_x + (a + by) \partial_y\, ,
$$
by Lemma \ref{lem1}. Since $[ \partial_y,\, X] \,= \, H'(y) \partial_x + b \partial_y$, it
now follows that $H$ is a polynomial of degree at most $k$.

$\textit{Case} \ 1$: The projection of $\mathfrak{g}$ on vector fields of the type $f(y) \partial_y$ is of dimension 2.

In this case a set of representatives of $\mathfrak{g}/\mathfrak{g}'$ is of the form
\begin{equation}\label{z1}
(\lambda x + \mu y^{k+1}) \partial_x + \tau y \partial_y\, .
\end{equation}
As $[\partial_x, \,(\lambda x + \mu y^{k+1})\partial_x + \tau y \partial_y] \,=\, \lambda \partial_x$, there must be
a representative of the form in \eqref{z1} with $\lambda \,\neq\, 0$. Thus we may assume that
$(x + \mu y^{k+1}) \partial_x + \tau y \partial_y$ is a representative of $\mathfrak{g}/\mathfrak{g}'$.
Assume first that $\tau \,\neq\, 0$; then we may take $\tau \,=\,1$. If, in general, we make a change of variables:
$$
\widetilde{x} = x + H(y), \ \ \widetilde{y} = y\, ,
$$
then $\partial_x = \partial_{\widetilde{x}}$ and $\partial_y \,=\, H'(\widetilde{y}) \partial_{\widetilde{x}} +
\partial_{\widetilde{y}}$.
Thus, in the coordinates $(\widetilde{x},\, \widetilde{y})$ there is a representative of the form
$\widetilde{x} \partial_{\widetilde{x}} + \widetilde{y} \partial_{\widetilde{y}}$.

Denote $\widetilde{x},\,\widetilde{y}$ by $x,\, y$ for the sake of convenience. Take any other representative
$(ax +by^{k+1})\partial_x + cy \partial_y$ $(here k \,\geq\, 1)$. Then
$$
[ x \partial_x + y \partial_y,\, (ax +by^{k+1})\partial_x + cy \partial_y] \,= \,bk y^{k+1} \partial_x\, .
$$
As $\mathfrak{g}'$ has only terms up to order $k$ in $\partial_x$, we must have $bk\,=\, 0$ and therefore $b \,=\, 0$.
Thus, the only possible representatives are of the form
$\lambda x \partial_x + \mu y \partial_y$ and hence $\mathfrak{g}$ is a subalgebra of
$$\widetilde{\mathfrak{g}} \,= \,\langle x \partial_x, \,y \partial_y \rangle + \mathfrak{g}'\, ,$$
and the commutators of $\widetilde{\mathfrak{g}}$ as well as of the algebra
$$
\langle a x \partial_x + y \partial_y \rangle + \mathfrak{g}', \ \ a \neq 0\, ,
$$
both coincide with $\mathfrak{g}'$. The argument when $\tau = 0$ is similar to the next case and we omit it here.

$\textit{Case} \ 2$: The projection of $\mathfrak{g}$ on vector fields of the type $f(y) \partial_y$ is one dimensional.

In this case every representative of $\mathfrak{g}/ \mathfrak{g}'$ is of the form
$(\lambda x + \mu y^{k+1})\partial_x$.
Consequently, $\mathfrak{g}$ is a subalgebra of the algebra
$\langle x \partial_x, y^{k+1} \partial_x \rangle + \mathfrak{g}'$.
Therefore, commutator of this algebra is clearly of rank 1. Thus, this case does not arise.
\end{proof}

\section{Solvable non-nilpotent algebras with abelian commutator}

Let $V(\mathbb{C}^{2})$ be the Lie algebra of all vector fields of the type
$$
P(x, y) \partial_x + Q(x, y) \partial_y\, ,
$$
where $P$ and $Q$ are $C^\infty$ functions defined on some open set of $\mathbb{C}^{2}$ and
$\mathfrak{g}$ be a finite dimensional solvable subalgebra of $V(\mathbb{C}^{2})$ which is not nilpotent
but its commutator $\mathfrak{g}'$ is abelian. Then, as seen in Lemma \ref{lem2}, either the rank of $\mathfrak{g}'$ is 1
or it is 2. In the latter case, $\langle \partial_x,\, \partial_y \rangle \,\subset\, \mathfrak{g}'$
in suitable coordinates
and hence $\mathfrak{g}$ is contained in the normalizer of $\langle \partial_x,\, \partial_y \rangle$. This is a
relatively straightforward case and the details are given in Section \ref{se6.1}.

\subsection{Solvable non-nilpotent algebras with commutator abelian of rank 2}\label{se6.1}

Assume first that the rank of $\mathfrak{g}'$ is two. In this case
$\mathfrak{g}' \,= \,\langle \partial_x,\, \partial_y \rangle$,
in suitable coordinates, and $\mathfrak{g}$ is a solvable subalgebra of the normalizer of $\mathfrak{g}'$ in
the Lie algebra of all vector fields in $\mathbb{R}$. The normalizer consists of all linear fields of the form
$$
(a_{11}x + a_{12}y + c_{1})\partial_x + (a_{21}x + a_{22}y + c_{2})\partial_y\, .
$$
Thus, $N(\mathfrak{g}')/ \mathfrak{g}'$ is isomorphic to $\mathfrak{gl}(2, \mathbb R)$ under the isomorphism
that sends any $A\, \in\, \mathfrak{gl}(2, \mathbb R)$ 
to the vector field $V_A$ on ${\mathbb R}^2$ defined by
$$
V_A\, :=\, (\partial_{x_1},\, \partial_{x_2})A\begin{pmatrix}
x_{1}\\ x_2
\end{pmatrix} \, .
$$
Note that if $D\, =\, C^{-1}AC$ and $\begin{pmatrix}
\widetilde{x}_{1}\\ \widetilde{x}_2
\end{pmatrix} \, =\,C^{-1}\begin{pmatrix}
{x}_{1}\\ {x}_2 \end{pmatrix}$,
where $C\, \in\, \text{GL}(2, \mathbb R)$, then we have
$$
V_{A} \,=\, (\partial_{\widetilde{x}_1},\, \partial_{\widetilde{x}_2})D
\begin{pmatrix}
\widetilde{x}_{1}\\ \widetilde{x}_2
\end{pmatrix}\, .
$$
This means that up to a point transformation, solvable algebras containing $\mathfrak{g}'$ have representatives in the maximal solvable subalgebras of $\mathfrak{gl}(2, \mathbb{R})$.
Up to conjugacy these are
$$
\langle
\begin{pmatrix}
a & b\\
0 & c
\end{pmatrix}
\, \Big\vert\, a, b, c \in \mathbb{R} \rangle \ \ \text{ or }
\langle
\begin{pmatrix}
a & 0\\
0 & a
\end{pmatrix},\
\begin{pmatrix}
0 & b\\
-b & 0
\end{pmatrix}\, \Big\vert\, a,\, b\, \in\, \mathbb R \rangle
$$
Thus, the maximal solvable subalgebras of $N(\mathfrak{g}')/ \mathfrak{g}'$ have representatives in
$$
\langle x \partial_x, y \partial_y \rangle\ \ \text{ or }\ \
\langle x \partial_x + y \partial_y, y \partial_x - x \partial_y \rangle\, .
$$
Finally keeping in mind that we want solvable subalgebras whose commutator is $\langle \partial_x, \partial_y \rangle$, we see that only the following types are possible:
\begin{enumerate}
\item $\langle x \partial_x, y \partial_y \rangle \ltimes \langle \partial_x, \partial_y \rangle$

\item $\langle x \partial_x + y \partial_y , y \partial_x - x \partial_y \rangle \ltimes \langle \partial_x, \partial_y \rangle$

\item $\langle x \partial_x + \lambda y \partial_y \rangle \ltimes \langle \partial_x, \partial_y \rangle$

\item $\langle \lambda(x \partial_x + y \partial_y) + y \partial_x - x \partial_y \rangle \ltimes \langle \partial_x, \partial_y \rangle$
\end{enumerate}

\subsection{Solvable non-nilpotent algebras with commutator abelian of rank 1}\label{secl}

Here we need to assume initially that we are working over $\mathbb{C}$.

\begin{lemma}\label{lem6.1}
Assume that $\mathfrak{g}$ is solvable and not nilpotent and $\mathfrak{g}'$ is
abelian of rank one. Then in suitable coordinates one of the following holds:
\begin{enumerate}
\item every element of $\mathfrak{g}$ is of the form $(a x + H(y)) \partial_{x}$;

\item or every element of $\mathfrak{g}$ is of the form
$(a x + H(y)) \partial_{x} + b \partial_{y}$
with at least one element whose coefficient $b$ is nonzero.
\end{enumerate}
\end{lemma}

\begin{proof}
By assumption,
$\mathfrak{g}'$ is of rank 1 and abelian, say
$$\mathfrak{g}' \,=\, \langle \phi_{1}(y) \partial_{x}, \,\cdots,\, \phi_{k}(y) \partial_{x} \rangle\, ,
$$
where $\{\phi_{i}\}_{1 \leq i \leq k}$ are linearly independent functions.
By Lie's theorem, $\mathfrak{g}$ fixes a line in $\mathfrak{g}'$ say $ \langle f(y) \partial_{x} \rangle$. Choose
variables $\widetilde{x}$ and $\widetilde{y}$ so that $f(y) \partial_{x} \,= \, \partial_{\widetilde{x}}$,
$y \,=\, \widetilde{y}$. Denote $\widetilde{x}$, $\widetilde{y}$ again by $x$, $y$ for convenience. We now have
$$
[ \partial_{x}, \,X]\, =\, \lambda \partial_{x}
$$
for all $X \,\in\, \mathfrak{g}$. Thus, if $X \,=\, P(x, y) \partial_{x} + Q(x, y) \partial_{y}$, then
$$
\frac{\partial P}{\partial x}\, =\, \lambda, \ \ \ \frac{\partial Q}{\partial x} \,=\, 0\, .
$$
Hence, every element of $\mathfrak{g}$ is of the form $(\lambda x + H(y)) \partial_{x} + Q(y) \partial_{y}$.
By a change of variables involving only $y$, we may suppose that the projection of $\mathfrak{g}$ on
fields of the type $f(y) \partial_{y}$ is a subalgebra of the algebra
$\langle (a + b y) \partial_{y},\, a,\, b \,\in\, \mathbb{C} \rangle$.
If this projection is of dimension 2, then both $\partial_{y}$ as well as $y \partial_{y}$ and hence
$[ \partial_{y},\, y \partial_{y}] \,=\, \partial_{y}$ will belong to the projection of $\mathfrak{g}'$. This is not possible because the rank of $\mathfrak{g}'$ is one.

Thus, either $\mathfrak{g}$ is of the form $(a x + H(y)) \partial_{x}$ or the projection of $\mathfrak{g}$ on vector fields of the type $f(y) \partial_{y}$ is of dimension 1. By a further linear change of variables involving only $y$, we may suppose the projection is $\langle \partial_{y} \rangle$. This completes the proof of the lemma.
\end{proof}

\begin{proposition}\label{propa}
If $\mathfrak{g}$ is a solvable and not nilpotent algebra and rank of $\mathfrak{g}$ is 1, then $\mathfrak{g}$ in suitable coordinates is equivalent to
$\mathfrak{g} \,= \, \langle x \partial_{x} \rangle + \mathfrak{g}'$
and $$\mathfrak{g}' \,=\, \langle \phi_{1}(y) \partial_{x},\, \cdots,\, \phi_{N}(y) \partial_{x} \rangle\, ,$$
where the $\phi_{i}$ are linearly independent $C^{\infty}$ functions.
\end{proposition}

\begin{proof}
As $\mathfrak{g}$ is of rank $1$, every element of $\mathfrak{g}$ is of the form $P(x, y) \partial_{x}$.
By repeating the argument in lemma \ref{lem6.1}, we may suppose every element fixes the line $\langle \partial_{\widetilde{x}} \rangle$ in $\mathfrak{g}'$. and
$$
\mathfrak{g}' \,=\, \langle \phi_{1}(\widetilde{y}) \partial_{\widetilde{x}},\, \cdots,\, \phi_{N}(\widetilde{y})
\partial_{\widetilde{x}} \rangle
$$
with $\phi_{1}(\widetilde{y}) \,=\, 1$.
Denoting $\widetilde{x}$, $\widetilde{y}$ again by $x$, $y$ we see that the equation
$[ \partial_{x},\, P(x, y) \partial_{x} ]\, =\, \lambda \partial_{x}$
implies $P \,=\, \lambda x + F(y)$. Hence, every element is of the form
$X \,=\, ( \lambda x + F(y)) \partial_{x}$.
If $\lambda = 0$ for all elements $X$, then $\mathfrak{g}$ will be abelian. Thus, we may suppose
that $(x + F(y)) \partial_{x}$ is in $\mathfrak{g}$. By the change of variables
$$
\widetilde{x} \,=\, x + F(y),\ \ \widetilde{y} \,=\, y\, ,
$$
we have $\partial_{x} \,=\, \partial_{\widetilde{x}}$ and $\partial_{y} \,=\, \partial_{\widetilde{y}} + F'(\widetilde{y}) \partial_{\widetilde{x}}$. Therefore, in these coordinates $\widetilde{x} \partial_{\widetilde{x}}$ is in $\mathfrak{g}$. Hence, if $X
\,=\, \lambda \widetilde{x} \partial_{\widetilde{x}} + F(\widetilde{y}) \partial_{\widetilde{x}} \in \mathfrak{g}$, then
$$[ \widetilde{x} \partial_{\widetilde{x}},\, X ] \,=\, - F(\widetilde{y}) \partial_{\widetilde{x}}\,\in\, \mathfrak{g}'\,.$$
Thus, $F(\widetilde{y})$ is in $\mathfrak{g}'$ and
$\mathfrak{g}/ \mathfrak{g}' \,=\, \langle \widetilde{x} \partial_{\widetilde{x}} \rangle$.
\end{proof}

\begin{theorem}\label{prop6.5}
Assume that
\begin{itemize}
\item $\mathfrak{g}$ is solvable and not nilpotent,

\item $\mathfrak{g}$ is of rank $2$ and

\item $\mathfrak{g}'$ is abelian and of rank $1$.
\end{itemize}
Then, in suitable coordinates, at least one of $\partial_{y}$ or $x\partial_{x}+\partial_{y}$
is in $\mathfrak{g}$.

Moreover, if $\partial_{y}$ is in $\mathfrak{g}$, then $\mathfrak{g}'$
is a sum of generalized eigenspaces of
$\partial_{y}$ of the form $e^{\lambda y}P(y) \partial_{x}$, where $P(y)$ are polynomials of degree
less than the multiplicity of the eigenvalue $\lambda$ of $\partial_{y}$ in
$\mathfrak{g}'$.

Similarly, if $x\partial_{x}+\partial_{y}$
is in $\mathfrak{g}$, then $\mathfrak{g}'$
is a sum of generalized eigenspaces of
$x\partial_{x}+\partial_{y}$ of the form $e^{(\lambda+1)y}P(y) \partial_{x}$, where $P(y)$ are polynomials of degree
less than the multiplicity of the eigenvalue $\lambda$ of $x\partial_{x}+\partial_{y}$ in
$\mathfrak{g}'$.
\end{theorem}

\begin{remark}
Precise description of $\mathfrak{g}/\mathfrak{g}'$ is given in Corollary \ref{cor6.4}.
\end{remark}

\begin{proof}[{Proof of Theorem \ref{prop6.5}}]
By Lemma \ref{lem6.1}, every element of $\mathfrak{g}$ in suitable coordinates is of the form
$$
(a x + H(y)) \partial_{x} + b \partial_{y}
$$
and there is at least one such element of $\mathfrak{g}$ in which $b \,\neq\, 0$ (as we are assuming that the rank of
$\mathfrak{g}$ is $2$). Consequently, we may suppose that $b \,=\,1$ for some element. Fix such an element as
\begin{equation}\label{efx}
X\, =\, (a x + H(y)) \partial_{x} + \partial_{y}\,.
\end{equation}
If in this element $a \,=\, 0$, then $X \,=\, H(y) \partial_{x} + \partial_{y}$ is in $\mathfrak{g}$.

If $H(y)\, =\, 0$, then $\partial_{y}$ is in $\mathfrak{g}$. Otherwise, the canonical form of $X$
is $\partial_{\widetilde{y}}$, where $\widetilde{x}\,=\, x+F(y)$, $\widetilde{y}\,=\, y$ and
$\frac{d F}{d y}\,=\, H$. In these coordinates,
$$
\partial_{x}\,=\, \partial_{\widetilde{x}}, \ \ \partial_{y}\,=\, \partial_{\widetilde{y}}+F'(\widetilde{y})
\partial_{\widetilde{x}}\, .
$$
Hence, in the variables $\widetilde{x},\, \widetilde{y}$, every element is of the same form as before.
Writing $\widetilde{x},\, \widetilde{y}$ as $x,\, y$ for convenience, we may assume that $\partial_{y}$
is in $\mathfrak{g}$. Thus, as $\mathfrak{g}'$ contains elements of the form $G(y)\partial_{x}$, the
generalized eigenspaces of $\partial_{y}$ in $\mathfrak{g}'$ are of the form $e^{\lambda y}P(y) \partial_{x}$,
where $P(y)$ are polynomials of degree
less than the multiplicity of the eigenvalue $\lambda$ of $\partial_{y}$ in
$\mathfrak{g}'$.

If in the element $X$ in \eqref{efx} we have $a\, \neq\, 0$, we may suppose, by scaling $y$, if necessary, that
$a\,=\, 1$. Therefore, we may suppose that $X\, =\, (x + H(y)) \partial_{x} + \partial_{y}$ is in $\mathfrak{g}$.
The arguments are similar to the case considered above: the change of variables
\begin{equation}\label{ecv}
\widetilde{x}\,=\, x+ e^y F(y),\ \ \widetilde{y}\,=\, y\, ,
\end{equation}
where $\frac{\partial F}{\partial y}\,=\,
-e^{-y}H(y)$, transforms $X$ to $\widetilde{x}\partial_{\widetilde{x}}+\partial_{\widetilde{y}}$ and 
$\partial_{x}$ to $\partial_{\widetilde{x}}$. In general, any change of variables
$\widetilde{x}\,=\, x+ K(y)$, $\widetilde{y}\,=\, y$ transforms $\partial_{x}$ to $\partial_{\widetilde{x}}$
and $\partial_{y}$ to $\partial_{\widetilde{y}}+K'(\widetilde{y})\partial_{\widetilde{x}}$.

Therefore, this change of variables in \eqref{ecv} leaves the form of $\mathfrak g$ unchanged. Writing
$\widetilde{x},\, \widetilde{y}$ as $x,\, y$ for convenience, we may suppose that $x\partial_{x}+\partial_{y}$
is in $\mathfrak g$. The invariant subspaces of $x\partial_{x}+\partial_{y}$
and $\partial_{y}$ in $\mathfrak{g}'$ are the
same. Consequently, $\mathfrak{g}'$ is a sum of generalized eigenspaces of
$x\partial_{x}+\partial_{y}$ of the form $e^{(\lambda+1)y}P(y) \partial_{x}$, where $P(y)$ are polynomials of degree
less than the multiplicity of the eigenvalue $\lambda$ of $x\partial_{x}+\partial_{y}$ in
$\mathfrak{g}'$.
\end{proof}

\noindent
\textbf{Notation.}\, In the following corollary $S$ is a finite set of complex numbers and
$n_\lambda$ is a positive integer for $\lambda\, \in\,S$; also, ${\mathfrak h}\,=\,\sum_{\lambda\, \in\,S}
V(\lambda)$, where $V(\lambda)$ consists of all vector fields of the type $e^{\lambda y} P(y)
\partial_{x}$ with ${\rm degree}(P)\, <\, n_{\lambda}$.

\begin{corollary}\label{cor6.4}
The Lie algebra $\mathfrak g$ in Theorem \ref{prop6.5} is one of the following forms:
\begin{enumerate}
\item $\mathfrak{g}\,=\, \langle \mathfrak{h},\, \partial_{y}\rangle$, if zero is not an eigenvalue of
$\partial_{y}$ operating on $\mathfrak{h}$.

\item $\mathfrak{g}\,=\, \langle \mathfrak{h},\, \partial_{y},\, x\partial_{x}\rangle$.

\item $\mathfrak{g}\,=\, \langle \mathfrak{h},\, \partial_{y},\, y^N\partial_{x}\rangle$, where $N$
is the multiplicity of the zero eigenvalue of $\partial_{y}$ operating on $\mathfrak{h}$.

\item $\mathfrak{g}\,=\, \langle \mathfrak{h},\, x\partial_{x} +\partial_{y}\rangle$,
if zero is not an eigenvalue of $x\partial_{x}+\partial_{y}$ operating on $\mathfrak{h}$.

\item $\mathfrak{g}\,=\, \langle \mathfrak{h},\, x\partial_{x} +\partial_{y},\, y^Ne^y\partial_{x}\rangle$,
where $N\, \geq\, 0$ is the multiplicity of the zero eigenvalue of $x\partial_{x}+\partial_{y}$ operating
on $\mathfrak{h}$.

\item $\mathfrak{g}\,=\, \langle \mathfrak{h},\, x\partial_{x} +\partial_{y},\,
x\partial_{x} +y^Ne^y\partial_{x}\rangle$,
where $N\, > \, 0$ is the multiplicity of the zero eigenvalue of $x\partial_{x}+\partial_{y}$ operating
on $\mathfrak{h}$.
\end{enumerate}
In all the above cases, $\mathfrak{h}\,=\, \mathfrak{g}'$.
\end{corollary}

\begin{proof}
We will indicate a proof only for the first three cases as the proofs for the remaining cases are
similar.

If $\partial_{y}\, \in\, \mathfrak g$ and $X$ is in $\mathfrak g$, then $[\partial_{y},\, X]\, \in\, 
{\mathfrak g}'$. Now $\partial_{y}$ is surjective on its generalized eigenspaces in ${\mathfrak g}'$
for nonzero eigenvalues, whereas, if $N$ is the multiplicity of the zero eigenvalue
of $\partial_{y}$ operating on $\mathfrak{g}'$, then $\partial_{y}$ and $y^N\partial_{x}$ generate the
generalized null eigenspace for $\partial_{y}$ operating on $\mathfrak{g}'$. This means that if zero is
not an eigenvalue of $\partial_{y}$ in $\mathfrak{g}'$, then $[\partial_{y},\, X]\,=\, [\partial_{y},\, Y]$
for some $Y\, \in\, \mathfrak{g}'$. Hence, $X-Y$ is in the centralizer of $\partial_{y}$ in the space of all
vector fields of the form $(a x + H(y)) \partial_{x} + b \partial_{y}$. Therefore, we have
$$
X-Y\,=\, C_1x\partial_{x} +c_2\partial_{y}+c_3\partial_{x}\, .
$$
Hence, $\mathfrak{g}/\langle \mathfrak{g}',\, \partial_{y}\rangle$ is of dimension at most two.

Now, $\mathfrak{g}/\langle \mathfrak{g}',\, \partial_{y}\rangle$ cannot be two dimensional, otherwise
$[x\partial_{x},\, \partial_{x}]\,=\, -\partial_{x}$ will be $\mathfrak{g}'$ and zero will be an
eigenvalue of $\partial_{y}$ in $\mathfrak{g}'$. This proves parts (1) and (2) in the corollary.

For part (3) we note that if zero is an eigenvalue of $\partial_{y}$ in $\mathfrak{g}'$ of
multiplicity $N$, then all vector fields of the type $P(y)\partial_{x}$, where $\text{degree}(P)\, \leq\, N$,
are mapped by $\partial_{y}$ to the generalized eigenspace for $\partial_{y}$ in $\mathfrak{g}'$.
The rest of the arguments are similar to those for parts (1) and (2).

For the remaining parts we need to compute the centralizer of $x\partial_{x}+\partial_{y}$
in the space of all vector fields of the form $(a x + H(y)) \partial_{x} + b \partial_{y}$. This centralizer
is spanned by $x\partial_{x},\, \partial_{y},\, e^y\partial_{x}$. The rest is similar to the cases
already considered.
\end{proof}

\begin{remark}\label{re6.4}
The results for the real case follow from the real version of Lie's theorem on solvable real algebras; namely such an 
algebra either fixes a line or a plane in any representation.
The reason is that if the complex line $\langle u+\sqrt{-1}v\rangle$ is invariant under action of the complexification of real solvable Lie
algebra $\mathfrak g$, then the subspace $\langle u,\,v\rangle$ is invariant under the action of $\mathfrak g$. (See \cite{AB}.)
In the case under consideration, for the representation 
of $\mathfrak g$ in ${\mathfrak g}'$ and ${\mathfrak g}'$ abelian of rank 1, one can choose coordinates so that
this line is $\langle\partial_x\rangle$ or the plane $\langle\partial_x,\, y \partial_x\rangle$.

This gives the classification in the real case, except that for complex eigenvalues, the generalized eigenvectors in 
the complexification of $\mathfrak g$ of the operators in Theorem \ref{prop6.5}
are to be replaced by the corresponding real form of
$({\mathfrak g}')^{\mathbb C}(\lambda) + ({\mathfrak g}')^{\mathbb C}(\overline{\lambda})$.
\end{remark}

\section*{Acknowledgements}

IB is supported by a J. C. Bose Fellowship. FM thanks the NRF of South Africa and the Abdus Salam School of 
Mathematical Sciences, Pakistan, for support in collaboration.

\end{document}